\documentclass[12points,reqno]{amsart}
\usepackage{amsfonts}
\usepackage{amssymb}
\usepackage{graphicx}
\usepackage{amsmath}

\newtheorem{theorem}{Theorem}
\newtheorem{lemma}{Lemma}

\newtheorem{corollary}{Corollary}
\theoremstyle{definition}
\newtheorem{definition}{Definition}
\newtheorem{example}{Example}

\theoremstyle{remark}

\numberwithin{equation}{section}

\everymath{\displaystyle}

\begin{document}
\title{ Centers of Cuntz-Krieger  $C^*$-algebras }

%

\author{Adel Alahmadi}
\address{Department of Mathematics, Faculty of Science, King Abdulaziz University, P.O.Box 80203, Jeddah, 21589, Saudi Arabia}
\email{analahmadi@kau.edu.sa}
\author{Hamed Alsulami}
\address{Department of Mathematics, Faculty of Science, King Abdulaziz University, P.O.Box 80203, Jeddah, 21589, Saudi Arabia}
\email{hhaalsalmi@kau.edu.sa}
%
%


\maketitle

\begin{abstract}
 For a finite directed graph $\Gamma$ we determine the center of the Cuntz-Krieger $C^*$-algebra $CK(\Gamma).$
\end{abstract}

\maketitle

\section{Definitions and  Main Results}

 Let $\Gamma =(V,E,s,r)$ be a directed graph that consists of a set $V$ of vertices and a set $E$ of edges, and two maps $r:E\rightarrow V$  and $s:E\rightarrow V$ identifying   the range and the source of
each  edge. A graph is finite if  both sets $V$ and $E$ are finite. A graph is row-finite if   $|s^{-1}(v)|<\infty $ for an arbitrary vertex $v\in V.$

\medskip

\begin{definition}
If $\Gamma$ is a row-finite graph, the Cuntz-Krieger $C^*$-algebra $CK(\Gamma)$ is the universal $C^*$-algebra generated ( as a $C^*$-algebra ) by $V\cup E$ and satisfying the relations
(1) $s(e)e=er(e)=e$ for all $e\in E,$ (2) $e^*f=\delta_{e,f}r(f)$ for all $e,f \in E,$ (3) $v=\sum\limits_{e\in s^{-1}(v)} ee^*$ whenever $s^{-1}(v)\neq\emptyset.$
\end{definition}

\medskip

The discrete $\mathbb{C}$-subalgebra of $CK(\Gamma)$ generated by $V, E, E^*$ is isomorphic to the Leavitt path algebra $L(\Gamma)$ of the graph $\Gamma$ ( see [T2] ). We will identify $L(\Gamma)$
with its image in $CK(\Gamma).$ Clearly $L(\Gamma)$ is dense in $CK(\Gamma)$ in the $C^*$- topology.

\medskip

A path is a finite sequence $p=e_{1}\cdots e_{n}$ of edges with $r(e_{i})=s(e_{i+1})$
for $1\leq i\leq n-1$. We consider the vertices to be paths of length zero.
\medskip
We let $Path(\Gamma)$ denote the set of all paths in the graph $\Gamma $ and extend the maps $r, s$ to $Path(\Gamma)$ as follows: for $p=e_{1}\cdots e_{n}$ we set
$s(p)=s(e_1), r(p)=r(e_n) .$ For $v\in V$ viewed as a path we set $s(v)=r(v)=v.$
\medskip
A vertex $w$ is a descendant of a vertex $v$ if there exists a
path $p\in Path(\Gamma )$ such that $s(p)=v$, $r(p)=w$.

A cycle is a path $C=e_{1}\cdots e_{n}, n\geq 1$ such that $s(e_1)=r(e_n)$ and all vertices $s(e_{1}), \ldots ,   s(e_{n})$ are distinct.
\medskip
An edge $e\in E$ is called an exit from the cycle $C$ if $s(e)\in \{s(e_1),\cdots, s(e_n)\}$, but $e\notin \{ e_1,\cdots, e_n\}$. A cycle without an exit is called a $NE$-cycle.
\medskip
A subset $W\subset V$ is  hereditary if all descendants of  an arbitrary vertex $w\in W
$ also lie in $W.$
\medskip
For two nonempty subsets $W_1, W_2 \subset V$ let $E(W_1, W_2)$ denote the set of edges
$\{e\in E \mid s(e)\in W_1, r(e)\in W_2\}.$

\bigskip
For a hereditary subset $W\subset V$ the $C^*$-subalgebra of $CK(\Gamma)$ generated by $W, E(W,W)$ is isomorphic to the Cuntz-Krieger algebra of the graph
\newline $(W, E(W,W), s|_{E(W,W)}, r|_{E(W,W)}),$ ( see [BPRS], [BHRS] ). We will denote it as $CK(W).$

\begin{example}\label{ex1}
 Let the graph $\Gamma$ be a cycle, $V=\{v_1,\cdots, v_n\},$ $E=\{e_1,\cdots,e_n\},$ $s(e_i)=v_i, \, 1\leq i\leq n;$ $r(e_i)=v_{i+1}$ for $1\leq i\leq n-1,\, r(e_n)=v_1.$
 The algebra $CK(\Gamma)$ in this case is isomorphic to the matrix algebra $M_n(T),$ where $T$ is the $C^*$-algebra of continuous functions on the unit circle.
 The center of $CK(\Gamma)$ is generated by the element $e_1\cdots e_n+e_2e_3\cdots e_1+\cdots+ e_ne_1\cdots e_1$ and is isomorphic $T.$
\end{example}
\medskip
For an arbitrary row-finite graph $\Gamma$ if a path $C=e_1\cdots e_n$ is a $NE$-cycle with the hereditary set of vertices $V(C)=\{s(e_1),\cdots, s(e_n)\}$
then we denote $CK(C)=CK(V(C))\cong M_n(T).$ The center $Z(C)$ of the subalgebra $CK(C)$ of $CK(\Gamma)$ is generated by the element
$$z(C)=e_1\cdots e_n+e_2e_3\cdots e_ne_1+\cdots+e_ne_1\cdots e_{n-1}.$$

We will need some definitions and some results from [AA1], [AA2].

\begin{definition}
Let $W\subset V$ be a nonempty subset. We say that a path $p=e_1\cdots e_n,\, e_i\in E, $ is an \textit{arrival path } in $W$ if $r(p)\in W,$ and $\{s(e_1),\cdots, s(e_n)\}\nsubseteq W.$
In other words, $r(p)$ is the first vertex on $p$ that lies in $W.$ In particular, every vertex $w\in W,$ viewed as a path of zero length,  is an arrival path in $W.$ Let $Arr(W)$ be the set of all arrival paths in $W.$
\end{definition}
\medskip
\begin{definition}
A hereditary set $W\subseteq V $ is called finitary if $|Arr(W)|$ $<\infty $.

\end{definition}

If $W$ is a hereditary finitary subset of $V$ then $e(W)=\sum\limits_{p\in Arr(W)} pp^*$ is a central idempotent in $L(\Gamma)$ and, hence in $CK(\Gamma)$, see [AA1],[AA2].
\medskip
If $C$ is a $NE$-cycle in $\Gamma$ with the hereditary set of vertices $V(C)$ then we will denote $Arr(C)=Arr(V(C)).$
If the set $V(C)$ is finitary then we will say that the cycle is finitary. In this case for an arbitrary element $z\in Z(C)$ the sum $\sum\limits_{p\in Arr(C)} pzp^*$ lies in the center of
$L(\Gamma)$ and $CK(\Gamma),$ see [AA2].

\begin{theorem}\label{thm1}
Let $\Gamma$ be a finite graph. The center $Z(CK(\Gamma))$ is spanned by:
\begin{itemize}
  \item [(i)] central idempotents  $e(W),$ where $W$ runs over all nonempty hereditary finitary subsets of $V$;
  \item [(ii)] subspaces $\{\sum\limits_{p\in Arr(C)}pzp^*\mid z\in Z(C)\},$ where $C$ runs over all finitary $NE$-cycles of $\Gamma.$
\end{itemize}
\end{theorem}
\medskip
\begin{corollary}
$Z(CK(\Gamma))$ is the closure of $Z(L(\Gamma)).$

\end{corollary}
\medskip
\begin{corollary}
The center $Z(CK(\Gamma))$ is isomorphic to a finite direct sum $\mathbb{C}\oplus\cdots\oplus\mathbb{C}\oplus T\oplus\cdots \oplus T.$
\end{corollary}
\medskip
In [CGBMGSMSH] it was shown that the center of a prime Cuntz-Krieger $C^*$-algebra is equal to $\mathbb{C} \cdot 1$ ( for a finite graph ) or to $(0)$ ( for infinite graph ).

We remark that for two distinct hereditary subsets $W_1\subsetneqq W_2$ the central idempotents $e(W_1), e(W_2)$ may be equal.

To make the statement of the Theorem more precise we will consider annihilator hereditary subsets.

Let $W$ be a nonempty subset of $V.$ Consider the subset $W^\perp =\{v\in V\mid \text{ the vertex } v \text{ does not have  descendants in } W\}.$

For empty set we let $\emptyset^\perp=V.$ It is easy to see that $W^\perp$ is always a hereditary subset of $V.$ If $W_1\subseteq W_2\subseteq V$ then
$W_1^\perp\supseteq W_2^\perp,\, (W^\perp)^\perp\supseteq W,\, ((W^\perp)^\perp)^\perp=W^\perp.$
\medskip
\begin{definition}

We will refer to $W^\perp,\, W\subseteq V$ as annihilator hereditary subsets.

\end{definition}
\medskip
In [AA1] it was proved that $(W^\perp)^\perp $ is the largest hereditary subset of $V$ such that every vertex in it has a descendant in $W$ and that $e(W)=e((W^\perp)^\perp).$

Hence in the part (i) of the Theorem we can let $W$ run over nonempty finitary hereditary annihilator subsets of $V.$

In fact there is a 1-1 correspondence ( and a Boolean algebra isomorphism ) between finitary hereditary annihilator subsets  of $V$ and central idempotents of $L(\Gamma)$
and $CK(\Gamma).$
\medskip
\begin{example}
Let $\Gamma=$ \includegraphics[width=.2\textwidth]{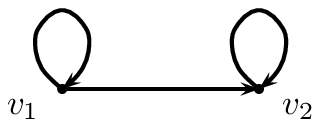} . The set $\{v_2\}$ is hereditary, but not finitary. Thus there are no proper finitary hereditary subsets and
$Z(L(\Gamma))=\mathbb{C}\cdot1.$
\end{example}
\medskip
\begin{example}
Let $\Gamma=$ \includegraphics[width=.1\textwidth]{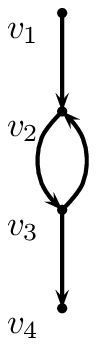} . The set $\{v_2,v_3,v_4\}$ is hereditary and  finitary, but $(\{v_2,v_3,v_4\}^\bot)^\bot=V.$ Thus there are no proper finitary hereditary annihilator subsets and again $Z(L(\Gamma))=\mathbb{C}\cdot1.$
\end{example}
\medskip
\begin{example}
Let $\Gamma=$ \includegraphics[width=.2\textwidth]{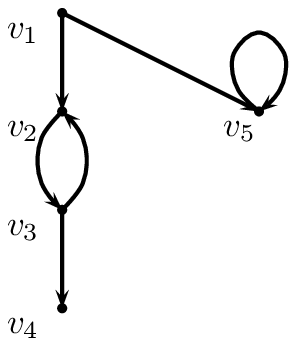} . The only finitary hereditary annihilator subsets are $\{v_5\}$ and  $\{v_2,v_3,v_4\}.$ Hence $Z(L(\Gamma))=T\oplus \mathbb{C}.$
\end{example}

\section{ Closed Ideals}

\bigskip

Let $W$ be a hereditary subset of $V.$ The ideal $I(W)$ of the Leavitt path algebra generated by the set W is the $\mathbb{C}$-span of the set $\{pq^*\mid p,q\in Path(\Gamma), r(p)=r(q)\in W\}.$ The closure $\overline{I}(W)$ of the ideal $W$ is the closed ideal of the $C^*$-algebra $CK(\Gamma)$ generated by the set $W.$

In this section we will use induction on $|V|$ to prove that for a proper hereditary subset $W\subset V$  central elements from $Z(CK(\Gamma))\cap \overline{I}(W)$ are of the form predicted by the Theorem.

\medskip

 Given paths $p,q\in Path(\Gamma)$ we say  $q$ is a \textit{continuation} of the path $p$ if there exists a path $p'\in Path(\Gamma)$ such that
$q=pp'.$ In this case the path $p$
 is called a \textit{beginning} of the path $q.$ We will often use the following well known fact:
 if $p,q\in Path(\Gamma)$, $p^*q\neq 0,$ then one of the paths $p,q$ is a continuation of the other one.
 
\medskip

 Remark that if $W$ is a hereditary subset of $V$ and $p,q\in Arr(W), p\neq q,$ then none of the paths $p,q$ is a continuation of the other one.
 
\medskip

 For a path $p\in Path(\Gamma)$ we consider the idempotent $e_p=pp^*.$

\begin{lemma}\label{lem1}
Let $W$ be a nonempty hereditary subset of $V$ let $a\in \overline{I}(W),$ let $\epsilon>0.$
Then the set $\{p\in Arr(W)\mid \|e_pa\|\geq\epsilon\}$ is finite.
\end{lemma}

\begin{proof}
The ideal $I(W)$ generated by the set $W$ in the Leavitt path algebra $L(\Gamma)$ is dense in $\overline{I}(W).$ Choose an element $b\in I(W),$
$b=\sum\limits_{p,q\in Arr(W)} pb_{p,q}q^*,$ $b_{p,q}\in L(W),$ such that $\|a-b\|<\epsilon.$ If $p'\in Arr(W)$ is an arrival path in $W$ that is different from all the
paths involved in the decomposition of the element $b$ then $p'^*b=0$ and $e_{p'}b=0.$

Now $$\|e_{p'}a\|=\|e_{p'}(a-b)\|\leq \|e_{p'}\|\|a-b\|<\epsilon,$$ which proves the Lemma.
\end{proof}

\begin{lemma}\label{lem2}
Let $a\in CK(\Gamma),$ $v\in V,$ $a\in vCK(\Gamma)v;$ $p\in Path(\Gamma), r(p)=v.$
Then $\|pap^*\|=\|a\|.$
\end{lemma}

\begin{proof}
We have $$\|pap^*\|\leq \|p\|\cdot\|a\|\cdot\|p^*\|=\|a\|.$$
On the other hand, $a=p^*(pap^*)p,$
 $$\|a\|\leq\|p^*\|\cdot\|pap^*\|\cdot\|p\|=\|pap^*\|,$$ which proves the Lemma.
\end{proof}

\begin{lemma}\label{lem3}
Let $W$ be a nonempty hereditary subset of $V$ let $z\in Z(CK(\Gamma)\cap\overline{I}(W).$ For an arbitrary vertex $w\in W$ if $wz\neq 0$
then the set $\{p\in Arr(W)\mid r(p)=w\}$ is finite.
\end{lemma}

\begin{proof}
Let $p\in Arr(W),r(p)=w, zw\neq0.$ We have $e_pz=pp^*z=pzp^*=p(wzw)p^*.$
By Lemma 2 $\|e_pz\|=\|wz\|>0.$ Now it remains to refer to Lemma 1.
\end{proof}

\medskip

Let $\widetilde{Z}$ denote the sum of all subspaces $\mathbb{C}e(W),$ where $W$ runs over nonempty finitary hereditary subsets of $V,$ and all subspaces
$\{\sum\limits_{p\in Arr(C)} pzp^*\mid z\in Z(C)\},$ where $C$ runs over finitary $NE$-cycles of $\Gamma.$

\medskip

Our aim is to show that $Z(CK(\Gamma))=\widetilde{Z}.$

\medskip

Let's use induction on the number of vertices. In other words, let's assume that for a graph with $<|V|$ vertices the assertion of the Theorem is true.

\medskip

\begin{lemma}\label{lem4}
Let $W$ be a proper hereditary subset of $V.$ Then $Z(CK(\Gamma)\cap\overline{I}(W)\subseteq\widetilde{Z}.$
\end{lemma}

\begin{proof}
Let $ 0\neq z\in Z(CK(\Gamma)\cap\overline{I}(W).$ Consider the element $z_0=z(\sum\limits_{w\in W} w)\in Z(CK(W)).$
If $z_0=0$ then $zW=(0), z\overline{I}(W)=(0),z^2=0,$ which contradicts semiprimeness of $CK(\Gamma)$ ( see [BPRS],[BHRS]).

By the induction assumption there exist disjoint hereditary finitary ( in $W$) cycles $C_1,\cdots, C_r$ and hereditary finitary (again in $W$) subsets
$W_1,\cdots, W_k\subset W$ such that $$z_0=\sum\limits_{i=1}^r \alpha_i \left(\sum\limits_{p\in Arr_{W}(C_i)}pa_ip^*\right)+\sum\limits_{j=1}^k \beta_j \left(\sum\limits_{q\in Arr_{W}(W_j)}qq^*\right); \alpha_i,\beta_j\in\mathbb{C}, a_i\in Z(C_i).$$

The notations $Arr_{W}(C_i), Arr_{W}(W_j)$ are used to stress that arrival paths are considered in the graph $(W,E(W,W)).$

The fact that the hereditary finitary subsets $V(C_i),W_j$ can be assumed disjoint follows from the description of the Boolean algebra of finitary hereditary subsets in [AA2].

If $\alpha_i\neq 0$ then for arbitrary vertex $w\in V(C_i)$ we have $z_0w=zw\neq0.$ Hence by Lemma 3 there are only finitely many paths $p\in Arr(W)$ such that $r(p)=w.$
Hence $V(C_i)$ is a finitary subset of $V.$ Similarly, if $\beta_j\neq 0$ then $W_j$ is a finitary subset in $V.$

Consider the central element $$z'=\sum\limits_{i=1}^r \alpha_i \left(\sum\limits_{p\in Arr(C_i)}pa_ip^*\right)+\sum\limits_{j=1}^k \beta_j e(W_j)\in Z(CK(\Gamma)).$$
We have $z'(\sum\limits_{w\in W} w)=z_0=z(\sum\limits_{w\in W} w).$ Hence, $(z-z')(\sum\limits_{w\in W} w)=(0),(z-z')\overline{I}(W)=(0), (z-z')^2=0.$
Again by semiprimeness of $CK(\Gamma)$ we conclude that $z=z',$ which proves the Lemma.

\end{proof}

\section{ Proof of the Theorem}

\begin{definition}

A vertex $v\in V$ is called a sink if $s^{-1}(v)=\emptyset.$

\end{definition}
\medskip
\begin{definition}

A hereditary subset $W\subset V$ is called saturated if for an arbitrary non-sink vertex $v\in V$ the inclusion $r(s^{-1}(v))\subseteq W$ implies $v\in W.$
\end{definition}
\medskip
\begin{definition}

If $W$ is hereditary subset then we define the saturation of $W$ to be the smallest saturated hereditary subset $\widehat{W}$ that contains $W.$ In this case $I(W)=I(\widehat{W}).$
\end{definition}
\medskip
\begin{definition}
If $W$ is a hereditary saturated subset of $V$ then the graph $\Gamma/W=(V\setminus W, E(V,V\setminus W))$ is called the factor graph of $\Gamma$ modulo $W.$

\end{definition}
We have $CK(\Gamma)/\overline{I}(W)\cong CK(\Gamma/W)$ ( see[T1] ).
\medskip
In [AAP], [T2] it was proved that the following 3 statements are equivalent:

\begin{itemize}
  \item [1)] the Cuntz-Krieger $C^*$-algebra $CK(\Gamma)$ is simple,
  \item [2)] the Leavitt path algebra $L(\Gamma)$ is simple,
  \item [3)] (i) $V$ does not have proper hereditary saturated subsets, (ii) every cycle has an exit.
\end{itemize}

We call a graph satisfying the condition 3) simple.

The following lemma is well known. Still we prove it for the sake of completeness.

\begin{lemma}\label{lem5}
Let $\Gamma$ be a graph such that $V$ does not have proper hereditary subsets. Then $\Gamma$ is either simple or a cycle.
\end{lemma}

\begin{proof}
If $\Gamma$ is not simple then $\Gamma$ contains a $NE$-cycle $C.$ The set of vertices $V(C)$ is hereditary subset of $V.$
In view of our assumption $V(C)=V,$ which proves the Lemma.

\end{proof}

Let $\mathrm{U}=\{\alpha\in\mathbb{C}\mid |\alpha|=1\}$ be the unit circle in $\mathbb{C}.$ Let $E'$ be a subset of the set $E$ of edges.
For an arbitrary $\alpha\in\mathrm{U}$ the mapping $g_{E'}(\alpha)$ such that $g_{E'}(\alpha):v\mapsto v, v\in V;$ $g_{E'}(\alpha):e\mapsto \alpha e, e^*\mapsto \overline{\alpha}e, e\in E';$ $g_{E'}(\alpha):e\mapsto  e, e^*\mapsto e^*, e\in E\setminus E',$ extends to an automorphism $g_{E'}(\alpha)$ of the $C^*$-algebra $CK(\Gamma).$
Denote $G_{E'}=\{g_{E'}(\alpha),\alpha\in\mathrm{U}\}\leq Aut(CK(\Gamma)).$
The group $G_E$ is called the gauge group of the $C^*$-algebra $CK(\Gamma).$ An ideal of $CK(\Gamma)$ is called gauge invariant if it is invariant with respect to the group $G_E.$

In [BPRS], [BHRS] it is proved that a nonzero closed gauge invariant ideal of $CK(\Gamma)$ has a nonempty intersection with $V.$

\begin{lemma}\label{lem6}
Let the graph $\Gamma$ be a cycle, $\Gamma=(V,E),\, V=\{v_1,\cdots,v_d\}, E=\{e_1,\cdots, e_d\},$ $s(e_i)=v_i$ for $1\leq i\leq d;$ $r(e_i)=v_{i+1}$ for $1\leq i\leq d-1,$ $r(e_d)=v_1.$ Then the central elements from $Z(CK(\Gamma))$ that are fixed by all $g_{E}(\alpha),\alpha\in\mathrm{U},$ are scalars.
\end{lemma}

\begin{proof}
The center of $CK(\Gamma)$ is isomorphic to the algebra of continuous function $T=\{f:\mathrm{U}\to \mathbb{C}\},$ the corresponding action of $G_E$ on $T$ is
$(g_{E}(\alpha)f)(u)=f(\alpha^du).$Now, if $f(u)=f(\alpha^du)$ for all $\alpha,u\in\mathrm{U}$ then $f$ is a constant function,  which proves the Lemma.

\end{proof}

\begin{lemma}\label{lem7}
Let $W$ be a hereditary saturated subset of $V$ such that $\Gamma/W$ is a cycle, $E(V\setminus W, W)\neq\emptyset,$ $\Gamma/W=\{v_1,\cdots,v_d\}.$ Then
$\left(\sum\limits_{i=1}^d v_iCK(\Gamma)v_i\right)\cap Z(CK(\Gamma))=(0).$
\end{lemma}

\begin{proof}
Consider the set $$W'=\{w\in W\mid E(V\setminus W,W)CK(\Gamma)w=(0)\}.$$
The set $W'$ is hereditary and saturated. Moreover, $\left(\sum\limits_{i=1}^d v_iCK(\Gamma)v_i\right)\cap \overline{I}(W')=(0).$
Indeed, we only need to notice that if $p\in Path(\Gamma),$ $s(p)=v_i$ and $r(p)\in W'$ then $p=0.$ Let $p=e_1\cdots e_n,$ $e_i\in E.$
At least one edge $e_j, 1\leq j\leq n,$ lies in $E(V\setminus W, W).$ This implies the claim.
Factoring out $\overline{I}(W')$ we can assume that $W'=\emptyset.$

Let $0\neq z\in\left(\sum\limits_{i=1}^d v_iCK(\Gamma)v_i\right)\cap Z(CK(\Gamma)).$ For an arbitrary edge $e\in E(V\setminus W, W)$ we have $ze=ez.$ Hence $ze=zer(e)=ezr(e)=0.$
Consider the ideal $$J=\{a\in CK(\Gamma)\mid E(V\setminus W,W)CK(\Gamma)a=aCK(\Gamma)E(V\setminus W,W)=(0)\}$$ of the algebra $CK(\Gamma).$ The element $z$ lies in $J.$ The ideal $J$ is gauge invariant.
Hence by [BPRS], [BHRS] $J\cap V\neq\emptyset.$ A vertex $v_i,\, 1\leq i\leq d,$ can not lie in $J$ because $v_iCK(\Gamma)E(V\setminus W,W)\neq (0).$ On the other hand
$J\cap W\subseteq W'=\emptyset,$ a contradiction, which proves the Lemma.

\end{proof}

\begin{lemma}\label{lem8}
Let $W$ be a hereditary saturated subset of $V$ such that $\Gamma/W$ is a cycle and  $E(V\setminus W, W)\neq\emptyset.$  Then
$Z(CK(\Gamma))\subseteq \mathbb{C}\cdot 1_{CK(\Gamma)}+\overline{I}(W).$
\end{lemma}

\begin{proof}
As in the Lemma 7 we assume that $V\setminus W=\{v_1,\cdots,v_d\}, E'=E\setminus E(V,W)=\{e_1,\cdots,e_d\};$ $s(e_i)=v_i, 1\leq i\leq d,$ $r(e_i)=v_{i+1}$ for $1\leq i\leq d-1,$ $r(e_d)=v_1.$
Consider the action of the group $G_{E'}$ on $CK(\Gamma).$ Let $z\in Z(CK(\Gamma)),$ $z=a+b, a=\sum\limits_{i=1}^d v_izv_i, b=\sum\limits_{w\in W} wzw.$ Since every element from $G_{E'} $ fixes $CK(W)$ it follows that $g(b)=b$ for all $g\in G_{E'}.$ Now $g(z)=g(a)+b, g(z)-z=g(a)-a\in \left(\sum\limits_{i=1}^d v_iCK(\Gamma)v_i\right)\cap Z(CK(\Gamma))=(0)$ by Lemma 7.
we proved that an arbitrary element of $Z(CK(\Gamma))$ is fixed by $G_{E'}.$
The ideal $\overline{I}(W)$ is invariant with respect to $G_{E'}.$ Hence the group $G_{E'}$ acts on $CK(\Gamma)/\overline{I}(W)\cong CK(\Gamma/W)$ as the full gauge group of automorphisms. The image of the central element $z$ in $Z(CK(\Gamma/W))$ is fixed by $G_{\Gamma/W}.$ Hence by Lemma 6 it is scalar. This finishes the proof of the Lemma.
\end{proof}

\begin{proof}[Proof of Theorem 1.]
If $V$ does not contain proper hereditary subsets then by Lemma 5 $\Gamma$ is either simple or a cycle.
If the $C^*$-algebra $CK(\Gamma)$ is simple then $Z(CK(\Gamma))=\mathbb{C}\cdot 1$ ( see [D] ) and the assertion of the Theorem is clearly true.

If $C$ is a cycle then $Z(CK(\Gamma))\cong T$ and the assertion of the Theorem is again true.

Let now $W$ be a maximal proper hereditary subset of $V.$ The saturation $\widehat{W}$ is equal to $V$ or the set $W$ is saturated.
In the first case $CK(\Gamma)=\overline{I}(W)$ and it suffices to refer to Lemma 4. Suppose now that the set $W$ is saturated. The graph $\Gamma/W$ does not have proper hereditary
subsets. Again by Lemma 5 the graph $\Gamma/W$ is either simple or a cycle.
If $\Gamma/W$ is simple then $Z(CK(\Gamma/W))=\mathbb{C}\cdot 1,$ which implies $Z(CK(\Gamma)\subseteq \mathbb{C}\cdot 1+\overline{I}(W),$ which together with Lemma 4 implies the Theorem. If $\Gamma/W$ is a cycle then by Lemma 8 we again have $Z(CK(\Gamma)\subseteq \mathbb{C}\cdot 1+\overline{I}(W)$ and the Theorem follows.

\end{proof}
\section*{Acknowledgement}
This project was funded by the Deanship of Scientific Research (DSR), King Abdulaziz University, under Grant No.
(27-130-36-HiCi). The authors, therefore, acknowledge technical and financial support of KAU.

\end{document}